\documentclass{amsart} %[11pt, a4paper, twoside, onecolumn]

\usepackage{amsmath,amssymb,amsthm,amsfonts,amsxtra,mathtools}
\usepackage{enumerate,verbatim,mathrsfs,comment,color,url}
\usepackage[all,2cell,ps]{xy}
\usepackage[pagebackref]{hyperref}
\usepackage{graphicx}
\usepackage{tikz-cd}
\usepackage{verbatim}
\usepackage{stackengine}

\DeclareMathOperator{\ann}{ann}
\DeclareMathOperator{\Ass}{Ass}

\DeclareMathOperator{\depth}{depth}
\DeclareMathOperator{\eend}{end}

\DeclareMathOperator{\indeg}{indeg}

\DeclareMathOperator{\reg}{reg}

\DeclareMathOperator{\Tor}{Tor}

\renewcommand{\ge}{\geqslant}
\renewcommand{\le}{\leqslant}

\newcommand{\bz}{\mathbb{Z}}

\newcommand{\fm}{\mathfrak{m}}
\newcommand{\fn}{\mathfrak{n}}
\newcommand{\fp}{\mathfrak{p}}
\newcommand{\fq}{\mathfrak{q}}
\newcommand{\fu}{\mathfrak{u}}

\newcommand{\lra}{\longrightarrow}

\theoremstyle{plain}
\newtheorem{theorem}{Theorem}[section]
\newtheorem{lemma}[theorem]{Lemma}
\newtheorem{proposition}[theorem]{Proposition}
\newtheorem{corollary}[theorem]{Corollary}

\theoremstyle{definition}
\newtheorem{definition}[theorem]{Definition}

\newtheorem{example}[theorem]{Example}

\newtheorem{notation}[theorem]{Notation}

\newtheorem{para}[theorem]{}

\newtheorem{setup}[theorem]{Setup}

\theoremstyle{remark}
\newtheorem{remark}[theorem]{Remark}

\numberwithin{equation}{section}

\title[On the asymptotic behaviour of the Vasconcelos invariant]{On the asymptotic behaviour of the Vasconcelos invariant for graded modules}
\author[L.~Fiorindo]{Luca Fiorindo}
\address{Dipartimento di Matematica, Dipartimento di Eccellenza 2023-2027, Università di Genova, Via Dodecaneso 35, 16146 Genova, Italy}
\email{luca.fiorindo@dima.unige.it}
\urladdr{\url{https://orcid.org/0000-0002-6435-0128}}

\author[D.~Ghosh]{Dipankar Ghosh}
\address{Department of Mathematics, Indian Institute of Technology Kharagpur, West Bengal - 721302, India}
\email{dipankar@maths.iitkgp.ac.in, dipug23@gmail.com}
\urladdr{\url{https://orcid.org/0000-0002-3773-4003}}

\date{May 25, 2024}
%; Secondary 13D07, 13B22

\begin{document}
\pagenumbering{arabic}
\thispagestyle{empty}
\begin{abstract}
	The notion of Vasconcelos invariant, known in the literature as v-number, of a homogeneous ideal in a polynomial ring over a field was introduced in \cite{CSTPV} to study the asymptotic behaviour of the minimum distance of projective Reed-Muller type codes. We initiate the study of this invariant for graded modules.	Let $R$ be a Noetherian $\mathbb{N}$-graded ring, and $M$ be a finitely generated graded $R$-module.	The v-number $v(M)$ can be defined as the least possible degree of a homogeneous element $x$ of $M$ for which $(0:_Rx)$ is a prime ideal of $R$. For a homogeneous ideal $I$ of $R$, we mainly prove that $v(I^nM)$ and $v(I^nM/I^{n+1}M)$ are eventually linear functions of $n$. In addition, if $(0:_M I)=0$, then $v(M/I^{n}M)$ is also eventually linear with the same leading coefficient as that of $v(I^nM/I^{n+1}M)$. These leading coefficients are described explicitly. The result on the linearity of $v(M/I^{n}M)$ considerably strengthens a recent result of Conca \cite{Conca} which was shown when $R$ is a domain and $M=R$, and Ficarra-Sgroi \cite{FS} where the polynomial case is treated.
\end{abstract}
\maketitle

\section{Introduction}
Recently, there has been some interest in the study of Vasconcelos invariants of homogeneous ideals in a polynomial ring over a field. The concept of Vasconcelos invariant appears in different areas of mathematics such as coding theory, algebraic geometry, and combinatorics. This numerical invariant is known in the literature as v-number, and has been named after the mathematician Wolmer Vasconcelos.
It was introduced in \cite{CSTPV} to express the regularity index of the minimum distance function of projective Reed-Muller type codes. In particular, in the same article, the authors connect the Vasconcelos invariant to the degree of projective varities consisting of finitely many points, see \cite[p.~16]{CSTPV}.
In combinatorics and graph theory, it has been showed a connection between the Vaconcelos invariant and the independent domination number. In \cite[3.5 and 3.6]{JV21}, the authors proved that the v-number of an edge ideal of a clutter corresponds to the independent domination number of the clutter itself. Thus, a combinatorial interpretation for the v-number of a square-free monomial ideal is provided.
% Moreover, in \cite[Thm.~3.20]{JV21}, they proved that the independent domination number of a whisker graph is the v-number of the corresponding edge ideal.

In this article, our aim is to extend the notion of v-numbers from homogeneous ideals in a polynomial ring to graded modules over a graded commutative Noetherian ring. Moreover, we interpret this invariant as the initial degree of certain graded module, and prove a number of results.

%Recently, there has been some interest in the study of Vasconcelos invariants of homogeneous ideals in a polynomial ring over a field. This numerical invariant is known in the literature as v-number. It was introduced in \cite{CSTPV} to express the regularity index of the minimum distance function, and has been named after the mathematician Wolmer Vasconcelos. See \cite[p.~16]{CSTPV} for a geometric description of v-numbers of ideals of certain projective varieties. A combinatorial interpretation is given in \cite[Thm.~3.5]{JV21} for the v-number of a square-free monomial ideal. In this article, our aim is to extend the notion of v-numbers from homogeneous ideals \new{in a polynomial ring} to graded modules over a graded commutative Noetherian ring. Moreover, we interpret this invariant as the initial degree of certain graded module, and prove a number of results.

\begin{setup}\label{setup}
	Throughout, unless specified otherwise, let $R = R_0[x_1,\ldots,x_d]$ be a commutative Noetherian $\mathbb{N}$-graded algebra, where $R_0$ denotes the $0$th graded component of $R$, and $\deg(x_i)\ge 1$ for $1\le i\le d$. Let $M$ be a finitely generated $\mathbb{Z}$-graded $R$-module, and $I$ be a homogeneous ideal of $R$. Let $N$ be a graded submodule of $M$ $($e.g., $N=\mathfrak{a}M$ for some homogeneous ideal $\mathfrak{a}$ of $R$$)$. Let $J$ be a reduction ideal of $I$ $($possibly, $J=I$$)$, and $J$ is generated by homogeneous elements $y_1,\ldots,y_c$ of degree $d_1\le\cdots\le d_c$ respectively.
\end{setup}

The set of associated prime ideals of the $R$-module $M$ is denoted by $\Ass_R(M)$. For $u\in\mathbb{Z}$, we usually write $M_u$ for the $u$th graded component of $M$. When $R$ is a polynomial ring over a field $R_0$, and $I$ is a homogeneous ideal of $R$, for each $\fp\in\Ass_R(R/I)$, the v-number of $I$ at $\fp$ is defined in \cite[Defn.~4.1]{CSTPV} as
\begin{equation}\label{v-num-ideal}
	v_\fp(I) := \inf\{ u\ge 0 : \mbox{there exists $f\in R_u$ such that } \fp = (I:_Rf) \}.
\end{equation}
The number $v(I) := \inf\{v_\fp(I) : \fp\in\Ass(R/I) \}$ is called the v-number of $I$. Generalizing this notion, we define v-number of $M$ as follows.

\begin{definition}\label{defn:v-num-M}
	For each $\fp\in\Ass_R(M)$, the v-number of $M$ at $\fp$ is the number $v_\fp(M) := \inf\{ u : \mbox{there exists $x\in M_u$ such that } \fp = (0:_Rx) \}$. Then, the v-number of $M$ is $v(M) := \inf\{v_\fp(M) : \fp\in\Ass_R(M) \}$. By convention, $v(0) = \infty$.
\end{definition}

\begin{remark}
	Note that each $\fp\in\Ass_R(M)$ is a homogeneous ideal of $R$. Moreover, $\fp = (0:_Rx)$ for some homogeneous element $x$ of $M$, see, e.g., \cite[1.5.6.(b)]{BH93}. Thus the invariants $v_\fp(M)$ and $v(M)$ in Definition~\ref{defn:v-num-M} are well-defined.
\end{remark}

\begin{remark}
	When $R$ is a polynomial ring over a field $R_0$, and $I$ is a homogeneous ideal of $R$, setting $M:=R/I$ in {\rm Definition~\ref{defn:v-num-M}}, one obtains $v_\fp(R/I)$ and $v(R/I)$, which are same as the numbers $v_\fp(I)$ and $v(I)$ respectively according to \cite[Defn.~4.1]{CSTPV}. Thus Definition~\ref{defn:v-num-M} recovers \cite[Defn.~4.1]{CSTPV}.
\end{remark}

We use the following notations frequently. With {\rm Setup~\ref{setup}},
\[
	(N:_MI) := \{x\in M:Ix\subseteq N\}, \quad \Gamma_I(M) := \bigcup_{n \ge 1}\big(0:_M I^n\big)% \mbox{ and } \ann_M(I) := (0:_M I).
\]
and $\ann_M(I) := (0:_M I)$. Denote $\indeg(M) := \inf\{n : M_n\neq 0 \}$ and $\eend(M) := \sup\{n : M_n\neq 0 \}$. By convention, $\indeg(0) =\infty$ and $\eend(0)=-\infty$. Inspired by \cite[Lem.~1.2]{Conca}, we interpret $v_{\fp}(M)$ as the initial degree of certain graded module as follows. It highly generalizes \cite[Prop.~4.2]{CSTPV} .
\begin{lemma}\label{lem:v-num-indeg}
	With {\rm Setup~\ref{setup}}, let $\fp\in\Ass_R(M)$. Set $X_\fp := \{\fq \in \Ass_R(M) : \fp \subsetneq \fq \}$. Let $V=R$ if $X_\fp=\emptyset$, otherwise $V = \prod_{\fq\in X_\fp}\fq$.
%	Denote
%	\[
%		\ann_M(\fp) := (0:_M\fp) \; \mbox{ and } \; \Gamma_V(M) := \bigcup_{n \ge 1}\big(0:_M V^n\big).
%	\]
	Then $$v_\fp(M) = \indeg\big(\ann_M(\fp)/\ann_M(\fp)\cap\Gamma_V(M)\big).$$
\end{lemma}

When $R$ is a polynomial ring over a field, for various classes of homogeneous ideals $I$, it has been shown that $v(R/I)\le\reg(R/I)$, see, e.g., \cite[Thm.~4.10]{CSTPV}, \cite[Thm.~3.13]{JV21}, \cite[p.~905]{SS22},
%\cite[Thm.~4.5]{AKS},
\cite[Thm.~3.8]{Ka23} and Remark~\ref{rmk:BM}. (Here $\reg(M)$ denotes the Castelnuovo-Mumford regularity of $M$, cf.~\ref{para:reg}). However, for each positive integer $n$, there are homogeneous ideals $I_n$ and $J_n$ for which $v(R/I_n)-\reg(R/I_n)=n$ and $\reg(R/J_n)-v(R/J_n)=n$, see \cite[Thm.~2]{Ci23} and \cite[Thm.~3.10]{Ka23} respectively.
With Setup~\ref{setup}, if $R$ is standard graded over a local ring $R_0$, then $v(M)\le\reg(M)$ whenever $\depth(\fm,M)=0$, where $\fm$ is the maximal homogeneous ideal of $R$, see Proposition~\ref{prop:comparison-gen} and Remark~\ref{rmk:v-reg}.
% \old{With Setup~\ref{setup}, under certain conditions, we observe that $v(M)\le\reg(M)$ (cf.~Proposition~\ref{prop:comparison-gen}).
% Thus, if $R$ is standard graded and $R_0$ is local, then $v(M)\le\reg(M)$ whenever $\depth(\fm,M)=0$, where $\fm$ is the maximal homogeneous ideal of $R$, see Remark~\ref{rmk:v-reg}.}
% Here $\depth(M)$ denotes the length of a maximal $M$-regular sequence.}
%Thus we get the following.
% , for graded modules of depth zero, the following holds.% (cf.~Remark~\ref{rmk:BM}).

% \begin{proposition}\label{prop:comparison-v-num-reg}
% 	With {\rm Setup~\ref{setup}}, let \new{$R$ be standard graded, $R_0$ is local}, and $\depth(M)=0$. Then $v(M)\le\reg(M).$
% 	In particular, $v(R/I)\le\reg(R/I)$ whenever $\depth(R/I)=0$.
% \end{proposition}

A classical result of Brodmann \cite{Br79} states that both the sets $\Ass_R(I^nM/I^{n+1}M)$ and $\Ass_R(M/I^{n}M)$ are eventually constants. Set $\mathcal{A}(I) := \Ass_R(R/I^{n})$ for all $n\gg 0$.
%Denote these stabilized sets by $\mathcal{B}_M(I)$ and $\mathcal{A}_M(I)$ respectively.
Recently, in \cite{Conca}, Conca proved that when $R$ is domain, for each $\fp\in\mathcal{A}(I)$, the function $v_\fp(R/I^{n})$ is eventually linear in $n$, i.e., $v_\fp(R/I^{n}) = an+b$ for all $n\gg 0$, where $a$ and $b$ are some constants. When $R$ is a polynomial ring over a field, this result has been shown independently by Ficarra-Sgroi in \cite[Thm.~3.1]{FS}. With Setup~\ref{setup}, the sets $\Ass_R(I^nM/I^nN)$ and $\Ass_R(M/I^nN)$ are also eventually constants due to McAdam-Eakin \cite[Prop.~2]{ME79} (cf.~\ref{para:ME-West})
%\footnote{McAdam-Eakin actually proved that if $\mathcal{R}$ is a standard graded Noetherian ring over $R$, then $\Ass_R(\mathcal{R}_n)$ is eventually constant. This proof can be modified to give the result for a finitely generated graded $\mathcal{R}$-module $\mathcal{M}$, see \cite[Thm.~3.4]{We04}. Finally, apply this result on the graded module $\mathcal{M} = \bigoplus_{n\ge 0}I^nM/I^nN$ over $\mathcal{R} = \bigoplus_{n\ge 0}I^n$.} 
and Katz-West \cite[Prop.~5.2]{KW04} respectively. So a natural question arises whether the functions $v(I^nM/I^nN)$ and $v(M/I^nN)$ are eventually linear in $n$? Another motivation of this question came from a result of Trung-Wang \cite[Thm.~3.2]{TW05} that $\reg(I^nM)$ is eventually a linear function of $n$. This was proved earlier for polynomial rings over a field by Cutkosky-Herzog-Trung \cite[Thm. 1.1]{CHT99} and Kodiyalam \cite{Ko00} independently.

\begin{notation}\label{notation}
	With Setup~\ref{setup},
%	in view of \cite[Prop.~2]{ME79} and \cite[Prop.~5.2]{KW04}, 
	we denote
	$$\mathcal{B}_N^M(I) := \Ass_R(I^nM/I^nN) \, \mbox{ and } \, \mathcal{A}_N^M(I) := \Ass_R(M/I^nN) \; \mbox{ for all }n\gg 0.$$
\end{notation}

Using Lemma~\ref{lem:v-num-indeg}, we prove Theorem~\ref{thm:main}, and deduce the following.

\begin{theorem}[See Theorem~\ref{thm:linearity-InM/InN} for more details]\label{thm:intro-lin-InM/InN}
    With {\rm Setup~\ref{setup}}
	and {\rm Notation~\ref{notation}}, let $\fp\in\mathcal{B}_N^M(I)$. Then, there exist $a\in\{d_1,\ldots,d_c\}$ and $b\in\mathbb{Z}$ such that $v_{\fp}(I^nM/I^nN) = an+b$ for all $n\gg 0$. Furthermore, both the functions
    \begin{center}
        $\indeg(I^nM/I^nN)$ \ and \ $v(I^nM/I^nN)$
    \end{center}
    are eventually linear in $n$ with the same leading coefficient $d_{\delta}\in \{d_1,\ldots,d_c\}$.
\end{theorem}

\begin{remark}
	In Theorem~\ref{thm:intro-lin-InM/InN}, one particularly may consider $N=0$ or $N=IM$. Note that Theorem~\ref{thm:intro-lin-InM/InN} is contained in Theorem~\ref{thm:linearity-InM/InN}, where the exact description of the leading coefficient $d_{\delta}$ is given, see \ref{para:delta}.
\end{remark}

If $L$ is a graded submodule of $M$, then $v(M)\le v(L)$, see Proposition~\ref{prop:L-M-sub}. Thus Theorem~\ref{thm:linearity-InM/InN} also provides a linear bound of $v(M/I^nN)$, cf.~Corollary~\ref{cor:lower-bdd}.

Our next theorem highly strengthens both the results \cite[Thms.~3.1 and 4.1]{FS} and \cite[Thm.~1.1]{Conca} of Ficarra-Sgroi and Conca in several directions. Note that $(0:_R I)=0$ when $R$ is a domain and $I$ is a non-zero ideal. Moreover, $M=R$, or $N=M$, or $J=I$ all are special cases in Setup~\ref{setup}.
%every ideal is a reduction ideal of itself.
%So Theorem~\ref{thm:Conca-gen} strengthens Conca's result in two directions???
%Our main result on asymptotic linearity is the following.

% \begin{customtheorem}{\ref{thm:Conca-gen}}\label{thm:intro-Conca-gen}
%     With {\rm Setup~\ref{setup}} and {\rm Notation~\ref{notation}}, let $(0:_M I)=0$.
% 	\begin{enumerate}[\rm (1)]
% 		\item Let $\fp\in\mathcal{A}_N^M(I)$ be such that $I\subseteq\fp$. Then, there exist $a\in\{d_1,\ldots,d_c\}$ and $b\in\mathbb{Z}$ such that $v_{\fp}(M/I^nN) = an+b$ for all $n\gg 0$. Moreover, if $\mathcal{B}_{IN}^M(I) = \mathcal{A}_N^M(I)$, then
%         $v_{\fp}(I^nM/I^{n+1}N) = v_{\fp}(M/I^{n+1}N) \mbox{ for all } n\gg 0.$
% 		\item Let $\mathcal{A}_N^M(I)\neq \emptyset$, and $I^{n_0}M\subseteq N$ for some $n_0$ $($e.g., if $N=M$, or if $N=\mathfrak{a}M$ for some homogeneous ideal $\mathfrak{a}$ satisfying $I \subseteq \sqrt{\mathfrak{a}}$$)$. Then, the functions
%         \[
%             \indeg\big( I^nM/I^{n+1}N \big), \;\; v\big( I^nM/I^{n+1}N \big) \; \mbox{ and } \; v\big( M/I^{n+1}N \big)
%         \]
%         all are eventually linear in $n$ with the same leading coefficient $d_{\gamma} \in \{d_1,\ldots,d_c\}$. In fact, the last two functions are asymptotically same when $\mathcal{B}_{IN}^M(I) = \mathcal{A}_N^M(I)$.
%         \item
%         When $(0 :_M y_1) = 0$ and $d_1\ge 1$, the leading coefficient in {\rm (2)} is $d_{\gamma} = d_1$.
%         % Additionally, if $(0 :_M y_1) = 0$, then $\gamma =  1$, i.e., the leading coefficient is $d_1$.
% 	\end{enumerate}
% \end{customtheorem}

\begin{theorem}[See Theorem~\ref{thm:Conca-gen} for more details]\label{thm:intro-Conca-gen}
With {\rm Setup~\ref{setup}} and {\rm Notation~\ref{notation}}, let $(0:_M I)=0$.
\begin{enumerate}[\rm (1)]
	\item Let $\fp\in\mathcal{A}_N^M(I)$ be such that $I\subseteq\fp$. Then, there exist $a\in\{d_1,\ldots,d_c\}$ and $b\in\mathbb{Z}$ such that $v_{\fp}(M/I^nN) = an+b$ for all $n\gg 0$. Moreover, if $\mathcal{B}_{IN}^M(I) = \mathcal{A}_N^M(I)$, then $v_{\fp}(I^nM/I^{n+1}N) = v_{\fp}(M/I^{n+1}N)$ for all $n \gg 0$.
	\item Let $\mathcal{A}_N^M(I)\neq \emptyset$, and $I^{n_0}M\subseteq N$ for some $n_0$ $($e.g., $N=M$, or $N=\mathfrak{a}M$ for some homogeneous ideal $\mathfrak{a}$ satisfying $I \subseteq \sqrt{\mathfrak{a}}$$)$. Then, the functions
    \[
        \indeg\big( I^nM/I^{n+1}N \big), \;\; v\big( I^nM/I^{n+1}N \big) \; \mbox{ and } \; v\big( M/I^{n+1}N \big)
    \]
    all are eventually linear in $n$ with the same leading coefficient $d_{\gamma} \in \{d_1,\ldots,d_c\}$. In fact, the last two functions are asymptotically same when $\mathcal{B}_{IN}^M(I) = \mathcal{A}_N^M(I)$.
    \item
    When $(0 :_M y_1) = 0$ and $d_1\ge 1$, the leading coefficient in {\rm (2)} is $d_{\gamma} = d_1$.
    % Additionally, if $(0 :_M y_1) = 0$, then $\gamma =  1$, i.e., the leading coefficient is $d_1$.
\end{enumerate}
\end{theorem}

Unlike \cite[Thm.~1.1]{Conca}, Theorem~\ref{thm:Conca-gen}.(2) describes the leading coefficients explicitly. Also it is quite surprising that the leading coefficients of the three linear functions in Theorem~\ref{thm:Conca-gen}.(2) are equal. Furthermore, whenever the modules $I^nM/I^{n+1}N$ and $M/I^{n+1}N$ have the same associate primes, their v-numbers are also same for every $n\gg 1$. In fact, these two numbers coincide for every $n\ge 1$ if in addition $(I^{n+1}N:_M I)=I^nM$ for all $n\ge 1$, see Remark~\ref{rmk:equality}.
Finally, in Section~\ref{sec:examples}, we construct some examples which complement our results. Among these, Example~\ref{example-1} particularly ensures that the hypothesis $(0:_M I)=0$ in Theorem~\ref{thm:Conca-gen} cannot be removed. Example~\ref{example-3} shows that the leading coefficient of the linear function $v(M/I^nM)$ is not necessarily same as $\indeg(I)$ even when $(0:_M I)=0$.
%By choosing suitable $M$ and $N$ in Theorem~\ref{thm:Conca-gen}, one produces various results on asymptotic linearity of v-numbers. One of these is given in the following corollary. For example $v(R/I^n+\fp)$ is eventually linear. 
%\begin{corollary}
%	With {\rm Setup~\ref{setup}}, let $R$ be a domain, and $\mathfrak{a}$ be a homogeneous ideal of $R$ satisfying $I \subseteq \sqrt{\mathfrak{a}}$. Then for each $\fp\in\mathcal{A}_N^M(I)$, there exist $a\in\{d_1,\ldots,d_c\}$ and $b\in\mathbb{Z}$ such that $v_{\fp}(R/I^n\mathfrak{a}) = an+b$ for all $n\gg 0$.
%\end{corollary}
\section{Main results and proofs}\label{sec:proofs}

In this section, we prove the results stated in the introduction. We start with the following, which is a generalization of the proof of  \cite[Lem.~1.2]{Conca}.

\begin{proof}[Proof of Lemma~\ref{lem:v-num-indeg}]
	Let $v = v_\fp(M)$ and $w = \indeg\big( \ann_M(\fp)/\ann_M(\fp)\cap\Gamma_V(M) \big)$. Then, by definition of $v_\fp(M)$, there exists a non-zero element $x \in M_v$ such that $\fp = (0 :_R x)$. Hence $\fp x=0$, i.e., $x \in \ann_M(\fp)$. We prove that $x\notin \ann_M(\fp)\cap\Gamma_V(M)$, equivalently, $x\notin \Gamma_V(M)$. If $V=R$, then $\Gamma_V(M)=0$, and there is nothing to prove. We may assume that $V\not=R$, i.e., $X_\fp\not=\emptyset$. If possible, assume that $x\in\Gamma_V(M)$. Then $V^mx=0$ for some integer $m\ge 1$. Thus $V^m \subseteq (0 :_R x) = \fp$, which implies that $V\subseteq\fp$. Hence, since $V = \prod_{\fq\in X_\fp}\fq$, it follows that $\fq\subseteq\fp$ for some $\fq\in X_\fp$. This is a contradiction because $\fp \subsetneq \fq$ (by the definition of $X_\fp$). So $x\notin \ann_M(\fp)\cap\Gamma_V(M)$. Thus the image of $x$ in $\ann_M(\fp)/\ann_M(\fp)\cap\Gamma_V(M)$ is a non-zero homogeneous element of degree $v$. Therefore $v\ge w$. It remains to prove the other inequality, i.e., $v\le w$.
	
	Suppose $y\in M_w$ induces a non-zero element of $\ann_M(\fp)/\ann_M(\fp)\cap\Gamma_V(M)$. Then $y\in\ann_M(\fp)\smallsetminus\Gamma_V(M)$. In particular, $\fp y=0$, i.e., $\fp\subseteq (0:_R y)$. If the equality holds, i.e., $\fp = (0:_R y)$, then $v\le w$. So it is enough to prove that $\fp = (0:_R y)$. If possible, assume that $\fp\subsetneq (0:_R y)$. Note that $(0:_R y)$ is a proper ideal, i.e., $\Ass_R(R/(0:_R y))\not=\emptyset$. Considering the map $R\to M$ given by $r\mapsto ry$, there is an injective $R$-module homomorphism $R/(0:_R y) \hookrightarrow M$. So $\Ass_R(R/(0:_R y)) \subseteq \Ass_R(M)$. Thus, for each $\fq\in\Ass_R(R/(0:_R y))$, one has that $\fp\subsetneq (0:_R y) \subseteq \fq \in \Ass_R(R/(0:_R y)) \subseteq \Ass_R(M)$, which yields that $\fq\in X_\fp$. Therefore $\Ass_R(R/(0:_R y)) \subseteq X_\fp$. In particular, $X_\fp\not=\emptyset$. Set $V_1 := \prod_{\fq\in \Ass_R(R/(0:_R y))}\fq$. Hence, from a primary decomposition of $(0:_R y)$, one deduces that there is an integer $m\ge 1$ such that $V_1^m \subseteq (0 :_R y)$, i.e., $V_1^m y = 0$. Since  $\Ass_R(R/(0:_R y)) \subseteq X_\fp$, it follows that $V\subseteq V_1$. So $V^m y \subseteq V_1^m y = 0$. Thus $y\in\Gamma_V(M)$, which is a contradiction. Therefore $\fp = (0:_R y)$.% This completes the proof.
\end{proof}

%Before proving Proposition~\ref{prop:comparison-v-num-reg}, we recall the notion of 
\begin{para}\label{para:reg}
	With Setup~\ref{setup}, further assume that $R$ is standard graded. Set $R_+:=\bigoplus_{n\ge 1}R_n$. The Castelnuovo-Mumford regularity of $M$ is given by
%	$\reg(M) := \max\{ \eend\big( H_{R_+}^i(M)\big) + i : 0\le i \le \dim(R) \}$,
	\[
		\reg(M) := \max\{ \eend\big( H_{R_+}^i(M)\big) + i : 0\le i \le \dim(R) \},
	\]
	where 
%	$R_+:=\bigoplus_{n\ge 1}R_n$, and 
	$H_{R_+}^i(M)$ denotes the $i$th local cohomology module of $M$ with respect to the ideal $R_+$. Note that $H_{R_+}^0(M) = \Gamma_{R_+}(M)$. It follows that $\eend\big( \Gamma_{R_+}(M) \big) \le \reg(M)$. Interested readers can look at \cite[Chapter 8]{BCRV22} and \cite{BCV22} for more details on this topic.
%	provided $\Gamma_{R_+}(M) \neq 0$. In the case, when $\fp\in\Ass_R(M)$ and $R_+\subseteq\fp$, then $\Gamma_{R_+}(M) \neq 0$. Indeed, if $R_+\subseteq\fp = (0:_Rx)$ for some non-zero $x\in M$, then $R_+x=0$, and hence $x\in\Gamma_{R_+}(M)$.
%	
%	let $R_0$ be a local ring with the maximal ideal $\fm$. Set $R_+:=\bigoplus_{n\ge 1}R_n$. Then $\fm\oplus R_+$ be the maximal homogeneous ideal of $R$. 	
%	
%	In addition, if we assume that $\depth(M)=0$, then $\reg(M) = \eend(\Gamma_{R_+}(M))$.
\end{para}

The v-number and regularity of a graded module can be compared as follows.

\begin{proposition}\label{prop:comparison-gen}
	With {\rm Setup~\ref{setup}}, assume that $R$ is standard graded. Denote $R_+:=\bigoplus_{n\ge 1}R_n$. If $\fp\in\Ass_R(M)$ for which $R_+\subseteq\fp$, then  $v_{\fp}(M)\le\reg(M)$. In particular, if $\depth(R_+,M)=0$, then $v(M)\le\reg(M)$.
%	\begin{enumerate}[\rm (1)]
%		\item If $\fp\in\Ass_R(M)$ be such that $R_+\subseteq\fp$, then  $v_{\fp}(M)\le\reg(M)$.
%		\item If $\depth(R_+,M)=0$, then $v(M)\le\reg(M)$.
%	\end{enumerate}
\end{proposition}

\begin{proof}
%	Let $\fp\in\Ass_R(M)$ and $R_+\subseteq\fp$. 
	Let $\fp\in\Ass_R(M)$ and $v = v_{\fp}(M)$. Then there exists $x \in M_v$ such that $\fp = (0 :_R x)$. Since $R_+\subseteq\fp$, it follows that $R_+x=0$. Thus $ x \in \Gamma_{R_+}(M) \cap M_v = \big(\Gamma_{R_+}(M)\big)_v $, which implies that $ v \le \eend\big( \Gamma_{R_+}(M) \big) \le \reg(M) $. This proves the first part. If $\depth(R_+,M)=0$, then $R_+ \subseteq \bigcup_{\fq\in\Ass_R(M)}\fq$, which yields that $R_+\subseteq\mathfrak{r}$ for some $\mathfrak{r}\in\Ass_R(M)$ (by prime avoidance), and hence $v(M)\le v_{\mathfrak{r}}(M)\le\reg(M)$.
\end{proof}

% \old{As an immediate consequence of Proposition~\ref{prop:comparison-gen}, one obtains Proposition~\ref{prop:comparison-v-num-reg}.
% \begin{proof}[Proof of Proposition~\ref{prop:comparison-v-num-reg}]
% 	Note that $\depth(M) = \depth(\fm,M)$, where $\fm$ is the maximal homogeneous ideal of $R$. Since $R_+\subseteq \fm$, $\depth(R_+,M) \le \depth(M)$. So, from the given condition, it follows that $\depth(R_+,M) =0$. Hence, by Proposition~\ref{prop:comparison-gen}, $v(M)\le\reg(M)$.	
% %	By $\depth(M)$, we mean $\depth(\fm,M)$, where $\fm$ is the maximal homogeneous ideal of $R$.
% %	If $\fn$ is the maximal ideal of $R_0$, and $R_+=\bigoplus_{n\ge 1}R_n$, then $\fm := \fn\oplus R_+$ is the maximal homogeneous ideal of $R$. Since $\depth(M)=0$, it follows that $\fm\in\Ass_R(M)$. Hence, by Proposition~\ref{prop:comparison-gen}, $v(M)\le\reg(M)$.
% \end{proof}
% }

\begin{remark}\label{rmk:v-reg}
    With {\rm Setup~\ref{setup}}, suppose $R$ is standard graded, and $(R_0,\fn)$ is local. Denote $R_+:=\bigoplus_{n\ge 1}R_n$. Then $R$ has the maximal homogeneous ideal $\fm:=\fn\oplus R_+$. Set $\depth(M) := \depth(\fm,M)$. Clearly, $\depth(R_+,M) \le \depth(M)$. Hence, by Proposition~\ref{prop:comparison-gen}, if $\depth(M)=0$, then $v(M)\le\reg(M)$. In particular, $v(R/I)\le\reg(R/I)$ whenever $\depth(R/I)=0$.
\end{remark}

\begin{remark}\label{rmk:BM}
	When $R$ is a polynomial ring over a field, and $I$ is a proper monomial ideal satisfying $\dim(R/I)=0$, it is shown in \cite[Thm.~4.19]{BM} that $v(R/I)\le\reg(R/I)$.
%	As $\depth(R/I)\le\dim(R/I)$,
	Note that Proposition~\ref{prop:comparison-gen} highly strengthens \cite[Thm.~4.19]{BM}.
%	ensures that the condition $\dim(R/I)=0$ can be replaced by a weaker condition that $\depth(R/I)=0$.
\end{remark}

From now on, we go back to the general case (Setup~\ref{setup}). The v-number of a module is less than or equal to that of any of its submodule.
%If $L$ is a graded submodule of $M$, then $v(M)\le v(L)$. In fact, we have the following.

\begin{proposition}\label{prop:L-M-sub}
	Let $L$ be a graded submodule of $M$. Then
	\begin{enumerate}[\rm (1)]
		\item $v_\fp(M) \le v_\fp(L)$ for each $\fp\in\Ass_R(L)$.
		\item $v(M) \le v(L)$.
	\end{enumerate}
\end{proposition}

\begin{proof}
	Let $\fp\in\Ass_R(L)$, and $v=v_\fp(L)$. Then there exists $x\in L_v$ such that $\fp = (0:_Rx)$. Since $x\in L_v \subseteq M_v$, one obtains that $\fp\in\Ass_R(M)$ and $v_\fp(M)\le v = v_\fp(L)$. It remains to prove the second part. Since $v(0)=\infty$, we may assume that $L\neq 0$, equivalently, $\Ass_R(L)$ is non-empty. Let $w=v(L)$. Then $w=v_\fq(L)$ for some $\fq\in\Ass_R(L)$. Since $\Ass_R(L)\subseteq\Ass_R(M)$, it follows that $\fq\in\Ass_R(M)$. By the first part, $v_\fq(M)\le v_\fq(L)=w$. Hence $v(M) \le v_\fq(M) \le w=v(L)$.
\end{proof}

\begin{para}\label{para:ME-West}
	In \cite[Prop.~2]{ME79}, McAdam-Eakin proved that if $\mathcal{R}$ is a Noetherian standard graded algebra over $R$, then $\Ass_R(\mathcal{R}_n)$ is eventually constant. This proof can be easily modified to give the result for a finitely generated graded $\mathcal{R}$-module $\mathcal{M}$, see \cite[Thm.~3.4]{We04}. Note that the Rees module $\mathscr{R}(I,M)$ is finitely generated graded over the Rees algebra $\mathscr{R}(I)$. Since $N$ is a submodule of $M$, the Rees module $\mathscr{R}(I,N)$ is a graded submodule of $\mathscr{R}(I,M)$. So the quotient $\mathscr{H} := \bigoplus_{n\ge 0}I^nM/I^nN$ is a finitely generated graded module over $\mathscr{R}(I)$. Thus, by \cite[Thm.~3.4]{We04}, it follows that the set $\Ass_R(I^nM/I^nN)$ is constant for all $n\gg 0$.
\end{para}

With Setup~\ref{setup}, we actually have a bigrading structure on $\mathscr{H} = \bigoplus_{n\ge 0}I^nM/I^nN$.

\begin{para}\label{para:Rees-ring-module}
	With Setup~\ref{setup}, let $\deg(x_i)=f_i$ for $1\le i\le d$. Let $\mathscr{R}(J) = \bigoplus_{n\in\mathbb{N}}J^n$ be the Rees algebra of $J$. We consider it as an $\mathbb{N}^2$-graded ring by setting the $(n,l)$th graded component of $\mathscr{R}(J)$ as the $l$th graded component of $J^n$ for each $(n,l)\in\mathbb{N}^2$. Then $\mathscr{R}(J)$ can be written as $ \mathscr{R}(J) = R_0[x_1,\ldots,x_d,y_1,\ldots,y_c] $, where $ \deg(x_i) = (0,f_i) $ for $ 1 \le i \le d $ and $ \deg(y_j) = (1,d_j) $ for $ 1 \le j \le c $. Consider the Rees module $\mathscr{R}(I,M) = \bigoplus_{n\in\mathbb{N}}I^nM $. By convention, $I^nM = 0$ whenever $n<0$. Setting $\mathscr{R}(I,M)_{(n,l)} :=(I^nM)_l$ for $(n,l)\in\mathbb{Z}^2$, we make $\mathscr{R}(I,M)$ a $\mathbb{Z}^2$-graded module over $\mathscr{R}(J)$. Since $J$ is a reduction ideal of $I$, the module $\mathscr{R}(I,M)$ is finitely generated over $\mathscr{R}(J)$. With similar gradation, $\mathscr{R}(I,N)$ is a $\mathbb{Z}^2$-graded $\mathscr{R}(J)$-submodule of $\mathscr{R}(I,M)$. So the quotient $\mathscr{H} := \mathscr{R}(I,M)/\mathscr{R}(I,N)$ is a finitely generated $\mathbb{Z}^2$-graded module over $\mathscr{R}(J)$.
\end{para}

We deduce Theorem~\ref{thm:linearity-InM/InN} from the following more general result.
%A counterpart of this theorem for regularity is shown in \cite{Gh16}.
%proposition, which is shown in \cite[Prop.~3.1]{CGN22}.

\begin{theorem}\label{thm:main}
    Let $ T = R_0 [x_1,\ldots,x_d,y_1,\ldots,y_c] $ be a bigraded ring over a commutative Noetherian ring $R_0$, where $ \deg(x_i) = (0,f_i) $ for $ 1 \le i \le d $ and $ \deg(y_j) = (1,d_j) $ for $ 1 \le j \le c $, where $d_1\le d_2\le \cdots\le d_c$. Let $ \mathscr{L} $ be a finitely generated $ \mathbb{Z}^2 $-graded $ T $-module. Set $\delta := \inf\left\{ j:y_j\notin\sqrt{\ann_T(\mathscr{L})}\right\}$. Set $ R := R_0[x_1,\ldots,x_d] $, where $ \deg(x_i) = f_i $ for $ 1 \le i \le d $. Denote $ \mathscr{L}_{(n,*)} := \bigoplus_{l\in\mathbb{Z}}  \mathscr{L}_{(n,l)}$ for each $n \in \mathbb{Z}$.
	
	Then, $ \mathscr{L}_{(n,*)} $ forms a $\mathbb{Z}$-graded $R$-module for each $n \in \mathbb{Z}$. Moreover, either $\mathscr{L}_{(n,*)} = 0$ for all $n\gg 0$, or $\mathscr{L}_{(n,*)} \neq 0$ for all $n\gg 0$. In the second case, $\delta$ is finite, and there exist $ a_{\fp} \in \{ d_j : \delta \le j \le c \} $ and $ b_1,b_2,b_{\fp} \in \mathbb{Z} $ such that
	\begin{enumerate}[\rm (1)]
		\item $\indeg\left(  \mathscr{L}_{(n,*)} \right) = d_{\delta} \cdot n + b_1 $ for all $ n \gg 0$,
		\item $v_{\fp}\left(  \mathscr{L}_{(n,*)} \right) = a_{\fp} \cdot n + b_{\fp}$ for all $ n \gg 0$ whenever $\fp\in\Ass_R\left(  \mathscr{L}_{(n,*)} \right)$ for all $n\gg 0$,
		\item $v\left(  \mathscr{L}_{(n,*)} \right) = d_{\delta} \cdot n + b_2 $ for all $ n \gg 0$.
	\end{enumerate}
%	\begin{equation*}
%		\indeg\left(  \mathscr{L}_{(n,*)} \right) = \indeg\big( \Tor_0^R( \mathscr{L}_{(n,*)},R_0) \big) = an + b \mbox{ for all } n \gg 0.
%	\end{equation*}
%	In $(1)$ and $(3)$, if $\mathscr{L}_{(n,*)} \neq 0$ for all $n\gg 0$, then both $b_1,b_2\in\mathbb{Z}$.
\end{theorem}

\begin{proof}
	Note that $x_i\mathscr{L}_{(n,*)}\subseteq \mathscr{L}_{(n,*)}$ for each $n \in \mathbb{Z}$ and $ 1 \le i \le d $. Thus,
	restricting the scalars from $T$ to $R$,
%	with the restricted scalar multiplications, 
	the set $ \mathscr{L}_{(n,*)} $ forms a $\mathbb{Z}$-graded $R$-module. From the construction of $ \mathscr{L}_{(n,*)} $, the bigraded module $ \mathscr{L} $ can be written as $ \mathscr{L} = \bigoplus_{n\in\mathbb{Z}} \mathscr{L}_{(n,*)} $. Since $y_j \mathscr{L}_{(n,*)} \subseteq \mathscr{L}_{(n+1,*)}$, we may consider $ \mathscr{L} = \bigoplus_{n\in\mathbb{Z}} \mathscr{L}_{(n,*)} $ as a $\mathbb{Z}$-graded module over an $\mathbb{N}$-graded ring $T=R[y_1,\ldots,y_c]$, where $ \deg(y_j) = 1 $ for $ 1 \le j \le c $, and $R$ is the $0$th graded component of $T$ in this gradation. Since we are only changing the grading (from bigraded to graded), $ \mathscr{L} = \bigoplus_{n\in\mathbb{Z}} \mathscr{L}_{(n,*)} $ is also finitely generated as a graded module over $T=R[y_1,\ldots,y_c]$. Consequently, by \cite[Thm.~3.4]{We04}, there exists $n_0$ such that $\Ass_R\left(  \mathscr{L}_{(n,*)} \right) = \Ass_R\left(  \mathscr{L}_{(n_0,*)} \right)$ for all $n\ge n_0$. Denote $\mathcal{A}_{\mathscr{L}} := \Ass_R\left( \mathscr{L}_{(n_0,*)} \right)$. Clearly, if $\mathcal{A}_{\mathscr{L}}$ is an empty-set, then $\mathscr{L}_{(n,*)} = 0$ for all $n\ge n_0$. In the other case, $\mathcal{A}_{\mathscr{L}}\neq \emptyset$, and we have that $\mathscr{L}_{(n,*)} \neq 0$ for all $n\ge n_0$. Since $\mathscr{L}_{(n,*)}\not =0$ for $n\gg 0$, we must have that $(y_1,\ldots,y_c)\not\subseteq \sqrt{\ann_T(\mathscr{L})}$, and hence $\delta$ is a finite number.
	
    (1) Consider an $\mathbb{N}^2$-graded polynomial ring $ S = R_0 [X_1,\ldots,X_d,Y_1,\ldots,Y_c] $ over $R_0$, where $ \deg(X_i) = (0,f_i) $ for $ 1 \le i \le d $ and $ \deg(Y_j) = (1,d_j) $ for $ 1 \le j \le c $. There is a natural graded ring homomorphism $S\to T$. Via this homomorphism, $\mathscr{L}$ is a finitely generated $ \mathbb{Z}^2 $-graded $ S $-module as well. Hence, by \cite[Prop.~3.1]{CGN22}\footnote{Note that \cite[Prop.~3.1]{CGN22} is proved using a result \cite[Thm.~4.6]{BCH13} of Bagheri-Chardin-H\`{a}. As \cite[Thm.~4.6]{BCH13} holds with this grading, we have the output of \cite[Prop.~3.1]{CGN22} in this setup as well.}, there exist $ a_1\in \{ d_j : 1 \le j \le c \} $ and $ b_1 \in \mathbb{Z} $ such that
	\begin{equation}\label{eqn:indeg-lin}
		\indeg\left(  \mathscr{L}_{(n,*)} \right) = \indeg\big( \Tor_0^R( \mathscr{L}_{(n,*)},R_0) \big) = a_1n + b_1 \mbox{ for all } n \gg 0.
	\end{equation}
    We show that $a_1 = d_{\delta}$, where $\delta = \inf\left\{ j:y_j\notin\sqrt{\ann_T(\mathscr{L})}\right\}$. As a graded module over $T=R[y_1,\ldots,y_c]$, let
    \begin{equation}\label{gen-le-n1}
        \mbox{$ \mathscr{L} = \bigoplus_{n\in\mathbb{Z}} \mathscr{L}_{(n,*)} $ be generated by homogeneous elements of degree $\le n_1$.}
    \end{equation}
    For $1\le j < \delta$, since $y_j \in \sqrt{\ann_T(\mathscr{L})}$, there exist $k_j$ such that $y_j^{k_j} \in \ann_T(\mathscr{L})$. Set $n_2 := (\delta-1)\cdot\max\{k_j : 1\le j < \delta\}$. So $(y_1,\dots,y_{\delta-1})^k\mathscr{L}=0$ for all $k\ge n_2$. Thus, for every $n\ge n_1+n_2$, one has that
    \begin{align*}
        \mathscr{L}_{(n,*)} &= \bigoplus_{i\le n_1} (y_1,\dots,y_c)^{n-i}\mathscr{L}_{(i,*)}\\
        &=\bigoplus_{i\le n_1} \ \bigoplus_{k \le n-i} (y_1,\dots,y_{\delta-1})^k(y_\delta,\dots,y_c)^{n-i-k}\mathscr{L}_{(i,*)}\\
        &=\bigoplus_{i\le n_1} \ \bigoplus_{k\le n_2} (y_1,\dots,y_{\delta-1})^k(y_\delta,\dots,y_c)^{n-i-k}\mathscr{L}_{(i,*)}.
    \end{align*}
    Moreover, for every $n \ge n_1+n_2$, note that $y_{\delta}\mathscr{L}_{(n,*)}\neq 0$. Indeed, if $y_{\delta}\mathscr{L}_{(n,*)}= 0$ for some $n \ge n_1+n_2$, then $y_{\delta}^{n-i+1}\mathscr{L}_{(i,*)} \subseteq y_{\delta}\mathscr{L}_{(n,*)} =0$ for all $i\le n_1$, and hence $y_{\delta} \in \sqrt{\ann_T(\mathscr{L})}$, a contradiction. Thus one concludes that
    \begin{equation}\label{eqn:indeg-dif}
        \indeg\left(  \mathscr{L}_{(n+1,*)} \right) - \indeg\left(  \mathscr{L}_{(n,*)} \right) = \deg(y_{\delta}) = d_{\delta} \mbox{ for all } n \gg 0.
    \end{equation}
    Consequently, the statement (1) follows from \eqref{eqn:indeg-lin} and \eqref{eqn:indeg-dif}.
    
	(2) Let $\fp\in\Ass_R\left(  \mathscr{L}_{(n,*)} \right)$ for all $n\gg 0$, equivalently, $\fp\in\mathcal{A}_{\mathscr{L}}$. Set $X_\fp := \{\fq \in \mathcal{A}_{\mathscr{L}} : \fp \subsetneq \fq \}$. Let $V=R$ if $X_\fp=\emptyset$, otherwise $V = \prod_{\fq\in X_\fp}\fq$. Let $\mathscr{M} = \ann_{\mathscr{L}}(\fp)/\ann_{\mathscr{L}}(\fp)\cap\Gamma_V(\mathscr{L})$. Note that $y_j \in \sqrt{\ann_T(\mathscr{L})} \subseteq \sqrt{\ann_T(\mathscr{M})}$ for all $1\le j < \delta$. Since $ T $ is Noetherian, and $\mathscr{L}$ is finitely generated, the (sub)quotient $\mathscr{M}$ is also a finitely generated $\mathbb{Z}^2$-graded module over $T$, where the grading of $\mathscr{M}$ is induced by that of $\mathscr{L}$. In particular, $ \mathscr{M}_{(n,*)} = \bigoplus_{l\in\mathbb{Z}}  \mathscr{M}_{(n,l)}$ is same as $\ann_{\mathscr{L}_{(n,*)}}(\fp)/\ann_{\mathscr{L}_{(n,*)}}(\fp)\cap\Gamma_V({\mathscr{L}_{(n,*)}})$. Therefore, in view of Lemma~\ref{lem:v-num-indeg} and (1), one has that $v_{\fp}\left( \mathscr{L}_{(n,*)} \right) = \indeg(\mathscr{M}_{(n,*)}) = a_{\fp}n+b_{\fp}$ for all $n\gg 0$, and for some $ a_{\fp} \in \{ d_j : \delta \le j \le c \} $ and $ b_{\fp} \in \mathbb{Z} $.
	
	(3) Note that $\mathcal{A}_{\mathscr{L}}$ is a (non-empty) finite set. By the definition of v-numbers, $v\left( \mathscr{L}_{(n,*)} \right) = \inf\{ v_{\fp}\left( \mathscr{L}_{(n,*)} \right) : \fp\in\mathcal{A}_{\mathscr{L}} \}$ for all $n\ge n_0$. Hence from (2) and the observation made in \ref{para:inf-lin-func}, the function $v\left(  \mathscr{L}_{(n,*)} \right) $ is eventually linear, i.e., $v\left(  \mathscr{L}_{(n,*)} \right) = a \cdot n + b_2 $ for all $ n \gg 0$, and for some $a \in \{ d_j : \delta \le j \le c \} $ and $ b_2 \in \mathbb{Z} $. Clearly, $d_{\delta} \le a$. It is enough to show that $a\le d_{\delta}$.
    
    Since the module $\mathscr{L}$ is Noetherian, the chain of submodules
    $$(0:_{\mathscr{L}} y_{\delta}) \subseteq (0:_{\mathscr{L}} y_{\delta}^2) \subseteq (0:_{\mathscr{L}} y_{\delta}^3) \subseteq \cdots $$
    stabilizes. So there exists $m_0\ge 1$ such that $(0:_{\mathscr{L}} y_{\delta}^m) = (0:_{\mathscr{L}} y_{\delta}^{m_0})$ for all $m\ge m_0$. In particular, $(0:_{\mathscr{L}_{(n_1,*)}} y_{\delta}^m) = (0:_{\mathscr{L}_{(n_1,*)}} y_{\delta}^{m_0})$ for all $m\ge m_0$, where $n_1$ is as in \eqref{gen-le-n1}. We prove that $(0:_{\mathscr{L}_{(n_1,*)}} y_{\delta}^{m_0})$ is a proper $R$-submodule of $\mathscr{L}_{(n_1,*)}$. If possible, let $(0:_{\mathscr{L}_{(n_1,*)}} y_{\delta}^{m_0})=\mathscr{L}_{(n_1,*)}$. Then $y_{\delta}^{m_0}\mathscr{L}_{(n_1,*)}=0$. Denote $\indeg(\mathscr{L}):=\inf\{n:\mathscr{L}_{(n,*)}\neq 0\}$. Setting $l:=\max\{0,-\indeg(\mathscr{L})\}$, for all $\indeg(\mathscr{L})\le n\le n_1$, one has that $y_{\delta}^{m_0+n_1+l}\mathscr{L}_{(n,*)}\subseteq y_{\delta}^{m_0+n+l}\mathscr{L}_{(n_1,*)} = 0$ as $n+l\ge 0$. Therefore, since $\mathscr{L}$ is generated by homogeneous elements of degree $\le n_1$, one concludes that  $y_{\delta}^{m_0+n_1+l}\mathscr{L}_{(n,*)}=0$ for all $n\in\bz$, and hence $y_{\delta}\in\sqrt{\ann_T \mathscr{L}}$, a contradiction. Thus $(0:_{\mathscr{L}_{(n_1,*)}} y_{\delta}^{m_0}) \subsetneq \mathscr{L}_{(n_1,*)}$.
    
    We show that $v\left( \mathscr{L}_{(n,*)} \right)\le v\Big( \mathscr{L}_{(n_1,*)}/ \big(0:_{\mathscr{L}_{(n_1,*)}} y_\delta^{m_0}\big)\Big)+(n-n_1)\cdot d_\delta$ for all $n\ge n_1+m_0$. For every $n\ge m\ge 0$, the natural map $\mathscr{L}_{(n-m-1,*)}(-d_{\delta}) \stackrel{y_{\delta}}{\lra} \mathscr{L}_{(n-m,*)}/\big(0:_{\mathscr{L}_{(n-m,*)}} y_\delta^{m}\big)$ induces a graded injective $R$-module homomorphism
    \begin{equation}\label{inj-hom}
        \dfrac{\mathscr{L}_{(n-m-1,*)}}{\big(0:_{\mathscr{L}_{(n-m-1,*)}} y_\delta^{m+1}\big)} (-d_{\delta}) \stackrel{y_{\delta}}{\lra} \dfrac{\mathscr{L}_{(n-m,*)}}{\big(0:_{\mathscr{L}_{(n-m,*)}} y_\delta^{m}\big)}.
    \end{equation}
    Here $M(-h)$ denotes a graded $R$-module with $M_{n-h}$ as its $n$-graded component. So, by definition of v-numbers, $v(M(-h)) = v(M) + h$. Thus, by Proposition~\ref{prop:L-M-sub}.(2), for every $n\ge m\ge 0$, the map \eqref{inj-hom} yields that
    \begin{equation}
        v\left( \dfrac{\mathscr{L}_{(n-m,*)}}{\big(0:_{\mathscr{L}_{(n-m,*)}} y_\delta^{m}\big)} \right) \le v\left( \dfrac{\mathscr{L}_{(n-m-1,*)}}{\big(0:_{\mathscr{L}_{(n-m-1,*)}} y_\delta^{m+1}\big)} \right) + d_{\delta}.
    \end{equation}
    Using this inequality repeatedly starting from $m=0$, one obtains that
    \begin{align*}
        v\left( \mathscr{L}_{(n,*)} \right) & \le v\left( \dfrac{\mathscr{L}_{(n-1,*)}}{\big(0:_{\mathscr{L}_{(n-1,*)}} y_\delta^{1}\big)} \right) + d_{\delta} \le v\left( \dfrac{\mathscr{L}_{(n-2,*)}}{\big(0:_{\mathscr{L}_{(n-2,*)}} y_\delta^{2}\big)} \right) + 2 \cdot d_{\delta} \le \cdots \\
        & \le v\left( \dfrac{\mathscr{L}_{(n_1,*)}}{\big(0:_{\mathscr{L}_{(n_1,*)}} y_\delta^{n-n_1}\big)} \right) + (n-n_1)\cdot d_{\delta} \\
        & = n \cdot d_{\delta} + v\Big( \mathscr{L}_{(n_1,*)}/ \big(0:_{\mathscr{L}_{(n_1,*)}} y_\delta^{m_0}\big)\Big) - n_1 \cdot d_{\delta} \;\mbox{ for all } n \ge n_1+m_0.
    \end{align*}
    It follows that $a\le d_{\delta}$, and hence $a= d_{\delta}$. This completes the proof.
\end{proof}

\begin{para}\label{para:inf-lin-func}
	Let $r$ be a positive integer. Let $a_i,b_i \in\mathbb{Z}$ for $i = 1,\ldots, r $. Define $f(n) = \inf\{ a_in+b_i : 1\le i\le r \}$ for all $n \in \mathbb{N}$. Then $f(n) = an+b$ for all $n\gg 0$, where $a := \inf\{ a_1,\ldots,a_r \}$ and $b:=\inf\{ b_i : a_i = a, 1\le i\le r \}$.
\end{para}

%Theorem~\ref{thm:linearity-InM/InN} follows from a more general result.
%
%\begin{theorem}
%Let $T$ and $R$ be as in Theorem~\ref{thm:main}. Let $ \mathscr{K} $ be a finitely generated $ \mathbb{Z}^2 $-graded $ T $-module. For each $n \in \mathbb{Z}$, denote $ \mathscr{K}_{(n,*)} := \bigoplus_{l\in\mathbb{Z}}  \mathscr{K}_{(n,l)}$.
%
%Then, for each $n \in \mathbb{Z}$, $ \mathscr{K}_{(n,*)} $ forms a $\mathbb{Z}$-graded $R$-module. Moreover, the set $ \Ass_R\left(\mathscr{K}_{(n,*)}\right) $ is eventually constant, say $\mathcal{A}_\mathscr{K}$. Furthermore, for each $\fp\in\mathcal{A}_\mathscr{K}$, $v_{\fp}\left(\mathscr{K}_{(n,*)}\right) = an+b$ for all $n\gg 0$, where $a\in\{d_1,\ldots,d_c\}$ and $b\in\mathbb{Z}$.
%\end{theorem}

Now we are in a position to prove Theorem~\ref{thm:linearity-InM/InN}. Here, in order to  describe the leading coefficient of $v(I^nM/I^nN)$, we need the following.

\begin{para}\label{para:delta}
    With {\rm Setup~\ref{setup}}, considering the (graded) module $\mathscr{H} := \mathscr{R}(I,M)/\mathscr{R}(I,N)$ over the Rees algebra $\mathscr{R}(J)$, set
    $\delta:=\inf\left\{j:y_j\notin\sqrt{\ann_{\mathscr{R}(J)}(\mathscr{H})}, 1 \le j \le c \right\}$.
\end{para}

\begin{theorem}\label{thm:linearity-InM/InN}
    With {\rm Setup~\ref{setup}} and {\rm Notation~\ref{notation}}, let $\mathcal{B}_N^M(I)$ be a non-empty set. Then, for every $\fp\in\mathcal{B}_N^M(I)$, there exist $a\in\{d_\delta,\ldots,d_c\}$ and $b\in\mathbb{Z}$ such that $v_{\fp}(I^nM/I^nN) = an+b$ for all $n\gg 0$, where $\delta$ is as in {\rm \ref{para:delta}}. Furthermore, both the functions $\indeg(I^nM/I^nN)$ and $v(I^nM/I^nN)$ are eventually linear in $n$ with the same leading coefficient $d_{\delta}\in \{d_1,\ldots,d_c\}$.
\end{theorem}

\begin{proof}
% [Proof of Theorem~\ref{thm:linearity-InM/InN}]
	With the discussion made in \ref{para:Rees-ring-module}, $\mathscr{H} := \mathscr{R}(I,M)/\mathscr{R}(I,N)$ is a finitely generated $\mathbb{Z}^2$-graded module over $ \mathscr{R}(J) = R_0[x_1,\ldots,x_d,y_1,\ldots,y_c] $, where $ \deg(x_i) = (0,f_i) $ for $ 1 \le i \le d $ and $ \deg(y_j) = (1,d_j) $ for $ 1 \le j \le c $. Note that $ \mathscr{H}_{(n,*)} := \bigoplus_{l\in\mathbb{Z}}  \mathscr{H}_{(n,l)}$ is given by $\mathscr{R}(I,M)_{(n,*)}/\mathscr{R}(I,N)_{(n,*)}$, which is same as $I^nM/I^nN$. Therefore, in view of Theorem~\ref{thm:main}, one deduces that:
    \begin{enumerate}[(1)]
        \item $\indeg(I^nM/I^nN)$ is eventually linear in $n$ with the leading coefficient $d_\delta$;
        \item for every $\fp\in\mathcal{B}_N^M(I)$, the function $v_\fp(I^nM/I^nN)$ is eventually linear in $n$ with the leading coefficient inside the set $\{d_\delta,\dots,d_c\}$;
        \item $v(I^nM/I^nN)$ is eventually linear in $n$ with the leading coefficient $d_\delta$.
    \end{enumerate}
    This completes the proof of the theorem.
\end{proof}
% \new{
% \begin{remark}
%     Let $\mathcal{I} := \{I_n\}_{n\in\mathbb{N}}$ be an $I$-admissible filtration of homogeneous ideals $I_n$ of $R$. Then the associated Rees algebra $\mathcal{R}(\mathcal{I}) := \bigoplus_{n\in\mathbb{N}}I_n$ is a finitely generated module over the Rees algebra $\mathcal{R}(I)$.
%     % i.e., $\{I_n\}_{n\in\mathbb{N}}$ is a filtration such that there exists $m\in\mathbb{N}$ satisfying $I^n\subseteq I_n \subseteq I^{n-m}$ for all $n\ge m$.
%     Hence a similar proof as that of Theorem~\ref{thm:linearity-InM/InN} provides that the functions $v_{\fp}(I_nM/I_nN)$ and $v(I_nM/I_nN)$ are asymptotically linear in $n$, where $\fp\in\Ass_R(I_nM/I_nN)$ for all $n\gg 0$.
% \end{remark}
% }
% \begin{remark}\label{def::delta}
%     \new{In the settings of Theorem \ref{thm:linearity-InM/InN}, looking at the structure of $\mathscr{H}$, one can see that
%     $$\delta=\inf\left\{j:\forall s\in\mathbb{N}, \exists \Bar{n}\in\mathbb{N} \text{ such that } I^{\Bar{n}}M\not\subseteq (I^{\Bar{n}+s}N :_M y_j^s)\right\}.$$
%     In particular, $\delta\ge \inf\left\{j:y_j\notin\sqrt{\ann_R M}\right\}$.}
% \end{remark}

As a consequence of Theorem~\ref{thm:linearity-InM/InN}, we obtain a linear bound of $v(M/I^nN)$.

\begin{corollary}\label{cor:lower-bdd}
    With {\rm Setup~\ref{setup}} and {\rm Notation~\ref{notation}}, let $\fp\in\mathcal{B}_N^M(I)$. Then $\fp\in\mathcal{A}_N^M(I)$, and there exist $a\in\{d_\delta,\ldots,d_c\}$ and $b\in\mathbb{Z}$ such that $v_{\fp}(M/I^nN) \le an+b$ for all $n\gg 0$, where $\delta$ is as in {\rm \ref{para:delta}}. Furthermore, $v(M/I^nN) \le d_\delta n+e$ for all $n\ge 1$, and for some $e\in\mathbb{Z}$.
%	, where $a\in\{d_1,\ldots,d_c\}$ and $e\in\mathbb{Z}$.
\end{corollary}

\begin{proof} 
% [Proof of Corollary~\ref{cor:lower-bdd}]
	Let $\fp\in\mathcal{B}_N^M(I)$. Since $I^nM/I^nN$ is a (graded) $R$-submodule of $M/I^nN$ for all $n\ge 0$, it follows that  $\mathcal{B}_N^M(I)\subseteq\mathcal{A}_N^M(I)$. So $\fp\in\mathcal{A}_N^M(I)$. In view of Theorem~\ref{thm:linearity-InM/InN}, there exist $a\in\{d_\delta,\ldots,d_c\}$ and $b\in\mathbb{Z}$
%	and $n_0\in\mathbb{N}$ 
    such that $v_{\fp}(I^nM/I^nN) = an+b$ for all $n\gg 0$. On the other  hand, by Proposition~\ref{prop:L-M-sub}.(1), $v_{\fp}(M/I^nN) \le v_{\fp}(I^nM/I^nN)$ for all $n\gg 0$. Combining these two results, $v_{\fp}(M/I^nN)\le an+b$ for all $n\gg 0$.
    % \old{Hence $v(M/I^nN)\le v_{\fp}(M/I^nN)\le an+b$ for all $n\gg 0$. So	there exists $n_0\in\mathbb{N}$ such that $v(M/I^nN)\le an+b$ for all $n\ge n_0$. Note that $M\neq IN$. Otherwise, if $M=IN$, then $M=I^nN$ for all $n\ge 1$, and hence $\mathcal{A}_N^M(I)$ is an empty set, a contradiction. So $M\neq IN$. Consequently, $M\neq I^nN$, and $v(M/I^nN)$ is finite for every $n\ge 1$. Considering $e$ as the maximum value among $b$ and $v(M/I^nN)$ for all $1\le n < n_0$, one gets that $v(M/I^nN) \le an+e$ for all $n\ge 1$.}
    % \old{On the other hand, by Proposition~\ref{prop:L-M-sub}.(2) and Theorem~\ref{thm:linearity-InM/InN}, for $n\gg 0$ we have $v(M/I^nN)\le v(I^nM/I^nN)=d_\delta n+e$ for some $e\in\mathbb{Z}$ and $\delta$ defined as in Remark \ref{def::delta}. Note that $M\neq IN$. Otherwise, if $M=IN$, then $M=I^nN$ for all $n\ge 1$, and hence $\mathcal{A}_N^M(I)$ is an empty set, a contradiction. So $M\neq IN$. Consequently, $M\neq I^nN$, and $v(M/I^nN)$ is finite for every $n\ge 1$.}

    For the second part, in view of Proposition~\ref{prop:L-M-sub}.(2) and Theorem~\ref{thm:linearity-InM/InN}, there exists $n_0$ such that $v(M/I^nN)\le v(I^nM/I^nN)=d_\delta n+b'$ for all $n\ge n_0$, and for some $b'\in\mathbb{Z}$. Note that $M\neq IN$. Otherwise, if $M=IN$, then $M=I^nN$ for all $n\ge 1$, and hence $\mathcal{A}_N^M(I)$ is an empty set, a contradiction. So $M\neq IN$. Consequently, $M\neq I^nN$, and $v(M/I^nN)$ is finite for every $n\ge 1$. Set $e$ as the maximum value among $b'$ and $\big(v(M/I^nN)-d_\delta n\big)$, $1\le n < n_0$. It follows that $v(M/I^nN) \le d_\delta n+e$ for all $n\ge 1$.
\end{proof}

%Modifying the proof of \cite[Thm.~1.1]{Conca}, we produce the following.
In the proof of Theorem~\ref{thm:Conca-gen}, we use the following lemma.

\begin{lemma}\label{lem:colon-comparison}
	With {\rm Setup~\ref{setup}}, let $(0:_M I)=0$. Let $\mathfrak{u}$ and $\mathfrak{a}$ be homogeneous ideals of $R$ such that $I\subseteq \mathfrak{u}$. Then, for all $n\gg 0$,
	\[
		\dfrac{\ann_{M/I^{n+1}N}(\mathfrak{u})}{\ann_{M/I^{n+1}N}(\mathfrak{u})\cap\Gamma_{\mathfrak{a}}(M/I^{n+1}N)} = \dfrac{\ann_{I^nM/I^{n+1}N}(\mathfrak{u})}{\ann_{I^nM/I^{n+1}N}(\mathfrak{u})\cap\Gamma_{\mathfrak{a}}(I^nM/I^{n+1}N)}.
	\]
%	as $\mathbb{Z}$-graded modules over $R$.
\end{lemma}

\begin{proof}
	Note that $(I^{n+1}N :_M \fu) \subseteq (I^{n+1}M :_M \fu) \subseteq (I^{n+1}M :_M I) = I^nM$ for all $n\gg 0$, where the last equality follows from \cite[Lem.~(4)]{Br79}. So
	\begin{equation*}
	(I^{n+1}N :_M \fu) = (I^{n+1}N :_M \fu) \cap I^nM  = (I^{n+1}N :_{I^nM} \fu) \mbox{ for all } n\gg 0.% \mbox{ and}
	\end{equation*}
	Going modulo $I^{n+1}N$ both sides, as graded submodules of $I^nM/I^{n+1}N$,
	%	 we have $\ann_{M/I^{n+1}N}(\mathfrak{u}) = \ann_{I^nM/I^{n+1}N}(\mathfrak{u})$ for all $n\gg 0$.
	%	 , and both are graded submodules of $I^nM/I^{n+1}N$.
	\begin{equation}\label{ann-same}
	\ann_{M/I^{n+1}N}(\mathfrak{u}) = \ann_{I^nM/I^{n+1}N}(\mathfrak{u}) \mbox{ for all } n\gg 0.% \mbox{ and}
	\end{equation}
	%	 Intersecting \eqref{ann-same} with 
	As $(I^nM/I^{n+1}N)\cap\Gamma_{\mathfrak{a}}(M/I^{n+1}N) = \Gamma_{\mathfrak{a}}(I^nM/I^{n+1}N)$, \eqref{ann-same} further induces that
	\begin{align}\label{ann-gamma-same}
	\ann_{M/I^{n+1}N}(\mathfrak{u})\cap\Gamma_{\mathfrak{a}}(M/I^{n+1}N) &=\ann_{I^nM/I^{n+1}N}(\mathfrak{u})\cap\Gamma_{\mathfrak{a}}(M/I^{n+1}N)\\
	&=\ann_{I^nM/I^{n+1}N}(\mathfrak{u})\cap\Gamma_{\mathfrak{a}}(I^nM/I^{n+1}N)\nonumber
	\end{align}
	for all $n\gg 0$. Combining \eqref{ann-same} and \eqref{ann-gamma-same}, one obtains the desired equalities.
\end{proof}

Now we give the following.

\begin{theorem}\label{thm:Conca-gen}
	With {\rm Setup~\ref{setup}} and {\rm Notation~\ref{notation}}, let $(0:_M I)=0$.
	\begin{enumerate}[\rm (1)]
		\item Let $\fp\in\mathcal{A}_N^M(I)$ be such that $I\subseteq\fp$. Then, there exist $a\in\{d_1,\ldots,d_c\}$ and $b\in\mathbb{Z}$ such that $v_{\fp}(M/I^nN) = an+b$ for all $n\gg 0$.
        Moreover, if $\mathcal{B}_{IN}^M(I) = \mathcal{A}_N^M(I)$, then
        $v_{\fp}(I^nM/I^{n+1}N) = v_{\fp}(M/I^{n+1}N)$ for all $n \gg 0$.
		\item Let $\mathcal{A}_N^M(I)\neq \emptyset$, and $I^{n_0}M\subseteq N$ for some $n_0$ $($e.g., $N=M$, or 
        $N=\mathfrak{a}M$ for some homogeneous ideal $\mathfrak{a}$ satisfying $I \subseteq \sqrt{\mathfrak{a}}$$)$. Then, the functions
        \[
            \indeg\big( I^nM/I^{n+1}N \big), \;\; v\big( I^nM/I^{n+1}N \big) \; \mbox{ and } \; v\big( M/I^{n+1}N \big)
        \]
        all are eventually linear in $n$ with the same leading coefficient $d_{\gamma} \in \{d_1,\ldots,d_c\}$, where $\gamma = \inf\left\{j:y_j\notin\sqrt{\ann_{\mathscr{R}(J)} (\mathscr{G})}, 1 \le j \le c \right\}$ and $\mathscr{G} := \mathscr{R}(I,M)/\mathscr{R}(I,IN)$. In addition, if $\mathcal{B}_{IN}^M(I) = \mathcal{A}_N^M(I)$, then
        $$v\big( I^nM/I^{n+1}N \big) = v\big( M/I^{n+1}N \big) \mbox{ for all } n \gg 0.$$
        \item When $(0 :_M y_1) = 0$ and $d_1\ge 1$, the leading coefficient in {\rm (2)} is $d_{\gamma} = d_1$.
	\end{enumerate}
\end{theorem}

\begin{proof}
% [Proof of Theorem~\ref{thm:Conca-gen}]
	 (1) Considering $IN$ in place of $N$ in \ref{para:Rees-ring-module}, one has that $\mathscr{G} = \mathscr{R}(I,M)/\mathscr{R}(I,IN)$ is a finitely generated $\mathbb{Z}^2$-graded module over $\mathscr{R}(J)$, where $\mathscr{R}(J)$ is an $\mathbb{N}^2$-graded ring with the same gradation as in \ref{para:Rees-ring-module}. Using the prime ideal $\fp\in\mathcal{A}_N^M(I)$, set $X_\fp := \{\fq \in \mathcal{A}_N^M(I) : \fp \subsetneq \fq \}$. Let $V=R$ if $X_\fp=\emptyset$, otherwise $V = \prod_{\fq\in X_\fp}\fq$. Let $\mathscr{L} = \ann_{\mathscr{G}}(\fp)/\ann_{\mathscr{G}}(\fp)\cap\Gamma_V(\mathscr{G})$. Then $\mathscr{L}$ is also a finitely generated $\mathbb{Z}^2$-graded module over $\mathscr{R}(J)$, where the bigrading in $\mathscr{L}$ is induced by that of $\mathscr{G}$. So $ \mathscr{L}_{(n,*)} = \bigoplus_{l\in\mathbb{Z}}  \mathscr{L}_{(n,l)}$ is same as $\ann_{\mathscr{G}_{(n,*)}}(\fp)/\ann_{\mathscr{G}_{(n,*)}}(\fp)\cap\Gamma_V({\mathscr{G}_{(n,*)}})$, where $\mathscr{G}_{(n,*)} = I^nM/I^{n+1}N$. Hence, since $I\subseteq\fp$, in view of Lemma~\ref{lem:colon-comparison},
	 $$ \mathscr{L}_{(n,*)} = \dfrac{\ann_{M/I^{n+1}N}(\mathfrak{p})}{\ann_{M/I^{n+1}N}(\mathfrak{p})\cap\Gamma_{V}(M/I^{n+1}N)} \mbox{ for all } n \gg 0.$$
    Finally, by Lemma~\ref{lem:v-num-indeg} and Theorem~\ref{thm:main}.(1), there exists an integer $b_1$ such that $v_{\fp}(M/I^{n+1}N) = \indeg(\mathscr{L}_{(n,*)}) = d_{\delta_\fp} n + b_1 $ for all $n \gg 0$, where
    \begin{equation}\label{delta-p}
        \delta_\fp = \inf\left\{j:y_j\notin\sqrt{\ann_{\mathscr{R}(J)} (\mathscr{L})}, 1 \le j \le c \right\}.
    \end{equation}
    Setting $b_\fp:=b_1-d_{\delta_\fp}$, one obtains that $v_{\fp}(M/I^nN) = d_{\delta_\fp} n+b_\fp$ for all $n\gg 0$. This proves the first part of (1). For the second part of (1), further assume that $\mathcal{B}_{IN}^M(I) = \mathcal{A}_N^M(I)$. Analyzing the proof above, one obtains that $v_{\fp}(I^nM/I^{n+1}N) = v_{\fp}(\mathscr{G}_{(n,*)}) = \indeg(\mathscr{L}_{(n,*)})=v_{\fp}(M/I^{n+1}N)$ for all $n\gg 0$, where the last two equalities follow from Lemma~\ref{lem:v-num-indeg} and the two different expressions of $\mathscr{L}_{(n,*)}$.
    
    (2) Note that $\mathcal{A}_N^M(I)$ is a finite non-empty set. Since $I^{n_0}M\subseteq N$ for some $n_0$, it can be verified that each $\fp\in\mathcal{A}_N^M(I)$ satisfies $I\subseteq\fp$. Hence, from the proof of (1), one observes that $\mathscr{G}_{(n,*)} = I^nM/I^{n+1}N\neq 0$ for all $n\gg 0$ (otherwise, if $\mathscr{G}_{(n,*)} = 0$, then $\mathscr{L}_{(n,*)} = 0$, and hence $v_{\fp}(M/I^{n+1}N)=\infty$ for all $n\gg 0$, a contradiction). Thus $\mathcal{B}_{IN}^M(I)$ is also a non-empty set. This will be used later while applying Theorem~\ref{thm:linearity-InM/InN}. Note that $v(M/I^{n+1}N) = \inf\{ v_{\fp}(M/I^{n+1}N) : \fp\in\mathcal{A}_N^M(I) \}$ for all $n\gg 0$. Therefore, since each $\fp\in\mathcal{A}_N^M(I)$ contains $I$, by (1) and \ref{para:inf-lin-func}, one concludes that $v(M/I^{n+1}N)$ is eventually linear in $n$ with the leading coefficient $d_\tau$, where $\tau:=\inf\{\delta_\fp: \fp\in\mathcal{A}_N^M(I)\}$ and $\delta_{\fp}$ is described in \eqref{delta-p}.
    % \old{We prove that $b=d_\gamma$, where $\gamma:=\inf\left\{j:y_j\notin\sqrt{\ann_{R} (M)}\right\}$.
    % First, note that if $y\in\sqrt{\ann_R (M)}$, then $y\in\sqrt{\ann_{\mathscr{R}(J)} (\mathscr{L})}$. So $\delta_{\fp}\ge\gamma$ for every $\fp\in\mathcal{A}_N^M(I)$. Hence $b\ge d_\gamma$. Secondly, in view of Proposition~\ref{prop:L-M-sub}.(2), $v(M/I^{n+1}N)\le v(I^nM/I^{n+1}N)$ for all $n\ge 0$. Moreover, the leading coefficient of the linear function $v(M/I^{n}N)$, and hence that of $v(M/I^{n+1}N)$ is same as $b$, while the leading coefficient of $v(I^nM/I^{n+1}N)$ is $d_\gamma$ by Proposition~\ref{prop:capitaldelta}.
    % % the leading coefficient of $v(M/I^{n+1}N)$, and hence of $v(M/I^nN)$ is bounded above by the leading coefficient of $v(I^nM/I^{n+1}N)$.
    % Thus, comparing the leading coefficients, it follows that $b\le d_\gamma$. Therefore $b = d_\gamma$.
    % }
    We prove that $\tau=\gamma$. First, notice that in the proof of (1), if $y_j\in\sqrt{\ann_{\mathscr{R}(J)} (\mathscr{G})}$, then $y_j\in\sqrt{\ann_{\mathscr{R}(J)} (\mathscr{L})}$. This yields that $\delta_{\fp}\ge\gamma$ for every $\fp\in\mathcal{A}_N^M(I)$. Hence $\tau\ge \gamma$. Secondly, in view of Proposition~\ref{prop:L-M-sub}.(2), $v(M/I^{n+1}N)\le v(I^nM/I^{n+1}N)$ for all $n\ge 0$. Here the leading coefficient of the asymptotic linear function $v(M/I^{n+1}N)$ is same as $d_{\tau}$, while the leading coefficients of $\indeg(I^nM/I^{n+1}N)$ and $v(I^nM/I^{n+1}N)$ are equal to $d_\gamma$ by Theorem~\ref{thm:linearity-InM/InN}. Thus, comparing the leading coefficients, it follows that $d_{\tau}\le d_\gamma$, which implies that $\tau\le \gamma$. So $\tau=\gamma$. It proves the first part of (2). Since each $\fp\in\mathcal{A}_N^M(I)$ contains $I$, the second part of (2) follows from (1).

    (3) 
    % Since $d_1\ge 1$, and $J$ is a reduction ideal of $I$, it follows that $I$ is generated by homogeneous elements of positive degree. So, if $\mathcal{A}_N^M(I)=\emptyset$, then $ M = I^nN \subseteq I^nM $, i.e.,  
    Assume that $(0 :_M y_1) = 0$ and $d_1\ge 1$. In view of (2), it is enough to show that $\gamma = 1$, i.e., $y_1 \notin \sqrt{\ann_{\mathscr{R}(J)} (\mathscr{G})}$. If  possible, let $y_1 \in \sqrt{\ann_{\mathscr{R}(J)} (\mathscr{G})}$. Then $y_1^s\mathscr{G}=0$ for some $s\ge 1$. Therefore, since $\mathscr{G}_{(n,*)} = I^nM/I^{n+1}N$, one obtains that $y_1^sI^nM \subseteq I^{n+s+1}N$ for all $n\ge 0$. Since $J$ is a reduction ideal of $I$, there exists $n_1$ such that $JI^{n_1} = I^{n_1+1}$, and hence $J^nI^{n_1} = I^{n_1+n}$ for all $n\ge 1$. Therefore, $y_1^sI^{n_1}M \subseteq I^{n_1+s+1}N=J^{s+1}I^{n_1}N$. Since $(0:_MI) = 0$, $I^{n_0}M\subseteq N$ and $M\neq 0$, it follows that $I^{n_1}N\neq 0$ and $I^{n_1}M\neq 0$. Moreover, since $J=(y_1,\ldots,y_c)$ and $(0 :_M y_1) = 0$, one derives that $\indeg(y_1^sI^{n_1}M) = sd_1+\indeg(I^{n_1}M)$ and $\indeg(J^{s+1}I^{n_1}N) = (s+1)d_1+\indeg(I^{n_1}N)$.
    % Since $(0:_MI) = 0$, $I^{n_0}M\subseteq N$ and $M\neq 0$, it follows that $I^{s+1}N_u \neq 0$, where $u:= \indeg(N)$. Moreover, $y_1^{s+1}N_u\neq 0$ as $(0 :_N y_1) \subseteq (0 :_M y_1) = 0$. So $\indeg(I^{s+1}N) = \indeg(I^{s+1})+\indeg(N) = (s+1)d_1+\indeg(N)$. Since $(0 :_M y_1) = 0$, $\indeg(y_1^sM) = sd_1+\indeg(M)$.
    Thus
    \begin{align*}
        sd_1+\indeg(I^{n_1}M) &= \indeg(y_1^sI^{n_1}M) \\
        &\ge \indeg(J^{s+1}I^{n_1}N) \quad \mbox{[as $y_1^sI^{n_1}M \subseteq J^{s+1}I^{n_1}N$]} \\
        &= (s+1)d_1+\indeg(I^{n_1}N) \\
        &\ge (s+1)d_1+\indeg(I^{n_1}M) \quad \mbox{[as $I^{n_1}N \subseteq I^{n_1}M$],}
        %\label{indeg-comparison}
    \end{align*}
    which is a contradiction as $d_1\ge 1$. So $y_1 \notin \sqrt{\ann_{\mathscr{R}(J)} (\mathscr{G})}$, and hence $d_\gamma = d_1$.
    % one has that $d_\gamma \ge d_1$. So it is enough to show that $d_\gamma \le d_1$. For $1\le m \le n-1$, the composition
    % $M(-d_1) \stackrel{y_1}{\longrightarrow} M \stackrel{\pi}{\longrightarrow} M/(I^nN :_M y_1^{m-1})$ induces an injective graded $R$-module homomorphism
    % $$ \dfrac{M}{(I^nN :_M y_1^m)}(-d_1) \stackrel{y_1}{\longrightarrow} \dfrac{M}{(I^nN :_M y_1^{m-1})} $$
    % Therefore, repeatedly applying Proposition~\ref{prop:L-M-sub} 
\end{proof}

Analyzing the proof of Theorem~\ref{thm:Conca-gen}, we make the following remarks.

% \new{
% \begin{remark}
%     Replacing $\{I^n\}$ by an $I$-admissible filtration $\{I_n\}$, a similar proof as that of Theorem~\ref{thm:linearity-InM/InN} provides that the functions $v_{\fp}(I_nM/I_nN)$ and $v(I_nM/I_nN)$ are asymptotically linear in $n$.
% \end{remark}
% }
\begin{remark}\label{rmk:equality}
    Let $n\ge 1$ be such that $\Ass_R(I^nM/I^{n+1}N) = \Ass_R(M/I^{n+1}N)$.
    \begin{enumerate}[(1)]
        \item In Theorem~\ref{thm:Conca-gen}.(1), $v_\fp(I^nM/I^{n+1}N)=v_\fp(M/I^{n+1}N)$ whenever $I\subseteq \fp\in \Ass_R(I^nM/I^{n+1}N)$ and $(I^{n+1}N:_M I)=I^nM$, because the equality of the quotients in Lemma~\ref{lem:colon-comparison} holds whenever $(I^{n+1}N:_M I) = I^nM$.
        \item Thus, in Theorem~\ref{thm:Conca-gen}.(2), one has that $v(I^nM/I^{n+1}N)=v(M/I^{n+1}N)$ whenever $(I^{n+1}N:_M I)=I^nM$ and $I^{n_0}M\subseteq N$ for some $n_0$.
    \end{enumerate}
\end{remark}

\section{Examples}\label{sec:examples}

Finally, we provide a number of examples to complement our results.

\begin{example}\label{example-1}
	Let $R=k[X,Y]$ be a standard graded polynomial ring in two variables over a field $k$. Set $M:=R/(XY^b)$, $I:=(X^a)$, $\fp:=(X)$, $\fq:=(Y)$ and $\fm := (X,Y)$, where $a$ and $b$ are some positive integers. Then, $(0:_M I) = \fq^b M \neq 0$. Moreover, the following hold true.
	\begin{enumerate}[(1)]
		\item $\Ass_R(I^nM) = \{ \fq \}$, $\indeg(I^nM/I^{n+1}M) = \indeg(I^nM) = an$ \ and
        \begin{center}
            $v(I^nM) = v_{\fq}(I^nM) = an+(b-1)$ for all $n\ge 1$.
        \end{center}
		\item $\Ass_R(M/IM) = \{ \fp \}$ if $a=1$, and $\Ass_R(M/I^nM) = \{ \fp,\fm \}$ whenever $an\ge 2$.
        \item
        $\Ass_R(I^nM/I^{n+1}M) = \{ \fm \}$ and $v(I^nM/I^{n+1}M) = an + (a+b-2)$ for all $n\ge 1$.
		\item $v_{\fp}(M/IM) = 0$ if $a=1$, and $v_{\fp}(M/I^nM) = b$ whenever $an\ge 2$.
		\item $v_{\fm}(M/I^nM) = an+(b-2)$ whenever $an\ge 2$.
		\item $v(M/IM) = 0$ if $a=1$, and $v(M/I^nM) = b$ whenever $an\ge 2$.
	\end{enumerate}
\end{example}

\begin{proof}
% [Proof of Example~\ref{example-1}]	
	Let $n\ge 1$. Then, $I^nM = (X^{an},XY^b)/(XY^b)$. It follows that
    $$\indeg(I^nM/I^{n+1}M) = \indeg(I^nM) = an.$$
    As $\Ass_R(I^nM) \subseteq \Ass_R(M) = \{ \fp,\fq \}$, and $(I^nM)_{\fp}=0$, one gets that $\Ass_R(I^nM) = \{ \fq \}$. Write the images of $X$ and $Y$ in $M$ as $x$ and $y$ respectively. The main relation of $x$ and $y$ that we have in $M$ is $xy^b=0$. So $(0:_M I) = \fq^b M \neq 0$. Note that each element of $I^nM$ can be written as $x^{an}f(x,y)$ for some polynomial $f(x,y)$ over $k$. Therefore $x^{an}y^{b-1} \in (I^nM)_{an+b-1}$ and $\fq = (0 :_R x^{an}y^{b-1})$. Clearly, $an+b-1$ is the least possible degree of a homogeneous element of $I^nM$ whose annihilator is $\fq$. It follows that $v(I^nM) = v_{\fq}(I^nM) = an+(b-1)$ for all $n\ge 1$. This proves (1).
	
    The quotient $M/I^nM\cong R/(X^{an},XY^b)$ for all $n\ge 1$. So $\Ass_R(M/I^nM) = \{ \fp,\fm \}$ if $an\ge 2$. If $a=1$, then $M/IM\cong R/(X)$, hence $\Ass_R(M/IM) = \{ \fp \}$ and $v_{\fp}(M/IM) = 0$. In the case, when $an\ge 2$, one has that $\fp = \big( (X^{an},XY^b) :_R Y^b \big)$ and $\fm = \big( (X^{an},XY^b) :_R X^{an-1}Y^{b-1} \big)$. These two equalities do not hold if $Y$ and $X^{an-1}Y^{b-1}$ are replaced respectively by any other homogeneous element of lower degree. Thus, one obtains (2), (4) and (5). Consequently, (6) follows.

    For (3), let $n\ge 1$. Note that $\Ass_R(I^nM/I^{n+1}M) \subseteq \Ass_R(M/I^{n+1}M) = \{\fp,\fm\} $ by (2). Therefore, since $(I^nM)_{\fp}=0$, one concludes that $\Ass_R(I^nM/I^{n+1}M) = \{ \fm \}$. Since $I^nM = (X^{an},XY^b)/(XY^b)$, the element $\overline{x^{an+a-1}y^{b-1}} \in I^nM/I^{n+1}M $ has the smallest possible degree such that $\fm =  \ann_R(\overline{x^{an+a-1}y^{b-1}})$. Therefore $v(I^nM/I^{n+1}M) = an + (a+b-2)$.
\end{proof}

\begin{remark}
	In Example~\ref{example-1}, we notice the following.
    \begin{enumerate}[(1)]
    \item The functions $v_{\fp}(M/I^nM)$ and $v(M/I^nM)$ of $n$ are eventually constants. Note that $(0:_M I) \neq 0$. It particularly ensures that the hypothesis $(0:_M I)=0$ in Theorem~\ref{thm:Conca-gen} cannot be removed.
    \item The functions $\indeg(I^nM/I^{n+1}M)$ and $v(I^nM/I^{n+1}M)$ (for all $n\ge 1$)
    % \begin{center}
    %     $\indeg(I^nM/I^{n+1}M) =  an$\ and $v(I^nM/I^{n+1}M) = an + (a+b-2)$\ for all $n\ge 1$
    % \end{center}
    both are linear with the same leading coefficient (as in Theorem~\ref{thm:linearity-InM/InN}), however their constant terms are different, namely $0$ and $(a+b-2)$ respectively. Moreover, the difference $(a+b-2)$ can be arbitrarily large depending on $a$ and $b$.
    \end{enumerate}
%	However, in this example, $v_{\fm}(M/I^nM)$ is eventually linear with the leading coefficient $a$.
\end{remark}

\begin{example}\label{example-2}
	Let $R=k[X,Y]/(XY)$ over a field $k$ with $\deg(X)=\deg(Y)=1$. Write the images of $X$ and $Y$ in $R$ as $x$ and $y$ respectively. Then $R=k[x,y]$. Set $\fm := (x,y)$ and $I:=(x^{d_1},y^{d_2})$, where $d_1 \le d_2$ are some positive integers. Then, $v(R/I^n) = v_{\fm}(R/I^n) = d_1 n - 1$ and $\reg(R/I^n) = d_2 n -1$ for every $n\ge 1$.
%	$\Ass_R(R/I^n) = \{ \fm \}$ for all $n\ge 1$.
%	\[
%		v(R/I^n) = v_{\fm}(R/I^n) = d_1 n - 1 \quad \mbox{and} \quad \reg(R/I^n) = d_2 n -1.
%	\]
%	\mbox{ for all } n\ge 1.
\end{example}

\begin{proof}
% [Proof of Example~\ref{example-2}]
	Let $n\ge 1$. Since $xy=0$, it follows that $I^n=(x^{d_1n},y^{d_2n})$. Therefore $R/I^n$ along with the gradation can be written as
	\[
	k \oplus (kx \oplus ky) \oplus (kx^2 \oplus ky^2) \oplus \cdots \oplus (k x^{d_1n-1} \oplus k y^{d_1n-1}) \oplus k y^{d_1n} \oplus\cdots \oplus k y^{d_2n-1}.
	\]
	Clearly, $\fm = (0 :_R x^{d_1n-1})$, and $d_1n-1$ is the least possible degree of a homogeneous element of $R/I^n$ whose annihilator is $\fm$. So $v(R/I^n) = v_{\fm}(R/I^n) = d_1 n - 1$. Since $R/I^n$ has finite length, $\reg(R/I^n) = \eend(R/I^n) = d_2n-1$.
\end{proof}

\begin{remark}
	Unlike
    % \cite[Thms.~3.1 and 4.1]{FS} and
    \cite[Thm.~1.1]{Conca}, Theorem~\ref{thm:Conca-gen} can be applied for a Noetherian graded ring which is not a domain. The ring $R$ in Example~\ref{example-2} is not a domain, however $(0:_RI)=0$, and $v(R/I^n)$ is linear with the leading coefficient $d_1$.
\end{remark}

\begin{remark}
    In Example~\ref{example-2}, if $d_1 < d_2$, then the difference $\reg(R/I^n)-v(R/I^n) = (d_2-d_1)n$ can be arbitrarily large depending on $n$.
\end{remark}

\begin{example}\label{example-3}
    Let $R=k[X,Y,Z]$ be a standard graded polynomial ring in three variables over a field $k$. Set $\fm := (X,Y,Z)$. Consider $M:=R/(X^3,XY^4)$ and $I:=(X,Y^2,Z^3)$. Then $(0:_M I)=0$. Moreover, the following hold true.
    \begin{enumerate}[(1)]
        \item $\Ass_R(I^nM/I^{n+1}M)=\Ass_R(M/I^nM) = \{ \fm \}$ for all $n\ge 1$.
        \item 
        $\indeg(I^nM/I^{n+1}M)=
        \left\{\begin{array}{ll}
            n &\text{if }n=0,1,2\\
            n+1 &\text{if }n=3\\
            n+3 &\text{if }n=4\\
            2n &\text{if }n\ge 5
        \end{array}\right.$
        \item
        $v(M/I^{n+1}M) = v(I^nM/I^{n+1}M) =
        \left\{\begin{array}{ll}
            n+3 &\text{if }n=0,1,2\\
            n+4 &\text{if }n=3\\
            n+6 &\text{if }n=4\\
            2n+3 &\text{if }n\ge 5
        \end{array}\right.$
    \end{enumerate}
\end{example}

\begin{proof}
    Since $M=R/(X^3,XY^4)$, the element $Z^3\in I$ is $M$-regular, and hence $(0:_M I)=0$. Note that $I^{n}M = \big(I^{n}+(X^3,XY^4)\big)/(X^3,XY^4)$, $M/I^{n}M  \cong R/\big(I^{n}+(X^3,XY^4)\big)$ and $I^{n}M/I^{n+1}M \cong \big(I^{n}+(X^3,XY^4)\big)/\big(I^{n+1}+(X^3,XY^4)\big)$ for all $n\ge 0$. We use $x,y,z$ for the classes of $X,Y,Z$ in $M$ respectively.
    
    (1) Let $n\ge 1$. Note that $I$ and $I^n$ annihilate $I^nM/I^{n+1}M$ and $M/I^nM$ respectively. Therefore every associated prime ideal of each of these modules contains $I=(X,Y^2,Z^3)$, and hence this prime ideal must be same as $\fm$.

    (2) 
    % We have that $X\in\sqrt{\ann_R M}$, $Y^2\notin\sqrt{\ann_R M}$ but it is not $M$-regular, and $Z^3$ is $M$-regular.
    A non-zero element of the least possible degree in the module $I^{n}M/I^{n+1}M$ for $0\le n \le 4$ is given by $1$, $x$, $x^2$, $x^2y^2$ and $x^2y^2z^3$ of degree $0$, $1$, $2$, $4$ and $7$ respectively. For $n\ge 5$, the module $I^nM$ is generated by
    $$x^2y^2(z^3)^{n-3},\ x^2(z^3)^{n-2},\ xy^2(z^3)^{n-2},\ x(z^3)^{n-1} \mbox{  and } (y^2)^{j}(z^3)^{n-j} \mbox{ for } 0\le j \le n,$$
    and their total degrees in $M$ are respectively
    \[
    3n-5, 3n-4, 3n-3,3n-2 \mbox{  and } 3n-j \mbox{ for } 0\le j \le n.
    \]
    Among these degrees, $2n$ is the least possible value. Thus $\overline{y^{2n}}$ is a non-zero element of the least possible degree in $I^nM/I^{n+1}M$ proving $(2)$.

    (3) Consider $n\ge 0$. By (1), $v(M/I^{n+1}M)=v_{\fm}(M/I^{n+1}M)$. Moreover, in view of Lemma \ref{lem:v-num-indeg}, one has that
    \begin{equation}\label{v-num-colon}
        v(M/I^{n+1}M)=v_{\fm}(M/I^{n+1}M)=\indeg \big((I^{n+1}M :_M \fm)/I^{n+1}M\big).
    \end{equation}

    \textbf{Claim}: We claim that $(I^{n+1}M:_M I)=I^nM$ for all $n\ge 0$. The claim is equivalent to that $\big(\big((X^3,XY^4)+I^{n+1}\big):_R I\big)=(X^3,XY^4)+I^n$ for all $n\ge 0$.
    
    For $n\ge 5$, the monomial ideal $((X^3,XY^4)+I^{n+1})$ is minimally generated by
    \begin{align*}
        X^3,XY^4, X^2Y^2(Z^3)^{n-2}, X^2(Z^3)^{n-1}, XY^2(Z^3)^{n-1}, X(Z^3)^{n}\\
        \mbox{and } (Y^2)^{j}(Z^3)^{n+1-j} \mbox{ for } 0\le j \le n+1.
    \end{align*}
    Therefore, using \cite[Sec.~3.2.2]{EJ}, one has that
    \begin{equation*}
        \begin{split}            \Big(\big((X^3,XY^4)+I^{n+1}\big):_RX\Big)=\Big( & X^2,Y^4,XY^2(Z^3)^{n-2}, X(Z^3)^{n-1}, \\
        &Y^2(Z^3)^{n-1},(Z^3)^{n}\Big),\\      \Big(\big((X^3,XY^4)+I^{n+1}\big):_RY^2\Big)=\Big(&X^3,XY^2,X^2(Z^3)^{n-2}, X(Z^3)^{n-1},\\
        (&Y^2)^{j}(Z^3)^{n-j},0\le j\le n\Big) \; \mbox{ and}\\            \Big(\big((X^3,XY^4)+I^{n+1}\big):_R Z^3\Big)=\Big(&X^3,XY^4,X^2Y^2(Z^3)^{n-3}, X^2(Z^3)^{n-2},\\
        XY^2(Z^3)^{n-2},&X(Z^3)^{n-1},(Y^2)^{j}(Z^3)^{n-j},0\le j\le n\Big).\\ 
        \end{split}
    \end{equation*}
    Now $\big(\big((X^3,XY^4)+I^{n+1}\big):_R I\big)$ is the intersection of the three ideals shown above. Moreover, the intersection of two monomial ideals is constructed by taking the lcm of pairs of generators one from each ideal. So the resulting ideal is exactly $(X^3,XY^4)+I^n$, which is minimally generated by 
    \begin{align*}
        X^3, XY^4, X^2Y^2(Z^3)^{n-3}, X^2(Z^3)^{n-2}, XY^2(Z^3)^{n-2}, X(Z^3)^{n-1}
        \mbox{ and } (Y^2)^{j}(Z^3)^{n-j}
    \end{align*}
    where $j$ is varying in $0\le j \le n$.
    
    The cases $n=0,\dots,4$ would require more attention, but instead they can be verified using any mathematical software (e.g., Macaulay2 \cite{M2}). Thus the claim is verified.
    
    Using the above claim, since $I\subseteq\fm$, it follows that
    \begin{equation}\label{colon-containment}
        (I^{n+1}M:_M\fm) \subseteq (I^{n+1}M:_MI) = I^{n}M \mbox{ for every } n\ge 0.
    \end{equation}
    For $n=0,1,\ldots,4$, using Macaulay2 \cite{M2}, one obtains that a non-zero homogeneous element in $(I^{n+1}M:_M\fm)/I^{n+1}M$ of minimum possible degree is given by
    \begin{center}
       $yz^2, xyz^2, x^2yz^2, x^2y^3z^2$ and $x^2y^3z^5$ respectively.
    \end{center}
    Hence, in view of \eqref{v-num-colon}, for $0\le n\le 4$, $v(M/I^{n+1}M)=3,4,5,7$ and $10$ respectively. Let $n\ge 5$. Then, since $\fm\subseteq (I^{n+1}M :_R y^{2n+1}z^2)$ and $y^{2n+1}z^2\notin I^{n+1}M$, one has that $\fm=(I^{n+1}M :_R y^{2n+1}z^2)$. Moreover, one can check that if $g\in I^{n}M$ with $\deg(g)<2n+3$, then $\fm\not=(I^{n+1}M:_R g)$. Therefore, using \eqref{v-num-colon} and \eqref{colon-containment}, it follows $v_{\fm}(M/I^{n+1}M)=2n+3$.

    In view of Remark~\ref{rmk:equality} and the above claim, one obtains that $v(I^nM/I^{n+1}M)=v(M/I^{n+1}M)$ for all $n\ge 0$.
\end{proof}

\begin{remark}
    In Theorem~\ref{thm:intro-Conca-gen}, the leading coefficient of the function $v(M/I^nN)$ is not necessarily same as $\indeg(I)$. In Example~\ref{example-3}, the leading coefficient of $v(M/I^nM)$ is $2$, however $\indeg(I)=1$. In this example, $X \in \sqrt{\ann_{R}(M)} \subseteq \sqrt{\ann_{\mathscr{R}(I)} (\mathscr{G})}$, but $Y^2 \notin\sqrt{\ann_{\mathscr{R}(I)} (\mathscr{G})}$, where $\mathscr{G} = \mathscr{R}(I,M)/\mathscr{R}(I,IM)$. It also ensures that the condition $(0 :_M y_1) = 0$ in Theorem~\ref{thm:intro-Conca-gen}.(3) cannot be removed.
\end{remark}

\section*{Acknowledgments}
The authors thank Aldo Conca for encouraging them to collaborate on the topic, for reading the draft, and giving his valuable comments and suggestions. The first author is partly supported by the MIUR Excellence Department Project awarded to Dipartimento di Matematica, Università di Genova, CUP\, D33C23001110001. The authors are grateful to the anonymous referee for carefully reading the first  draft of the article, and providing many helpful comments and suggestions.
%The author was supported by Core Research Grant (CRG) from SERB, DST, Govt.~of India with the Grant No CRG/2023/001224.
%Saha was supported by Junior Research Fellowship (JRF) from UGC, MHRD, Govt.\,of India.

\end{document}